\documentclass{amsart}
\usepackage[latin9]{inputenc}
\usepackage{amsthm}
\usepackage{amstext}
\usepackage{amssymb}
\usepackage{esint}
\usepackage[unicode=true,
 bookmarks=true,bookmarksnumbered=false,bookmarksopen=false,
 breaklinks=false,pdfborder={0 0 1},backref=false,colorlinks=false]
 {hyperref}
\hypersetup{pdftitle={H\"older regularity for Maxwell's equations under minimal assumptions on the coefficients},
 pdfauthor={Giovanni S. Alberti}}

\makeatletter
\theoremstyle{plain}
\newtheorem{thm}{\protect\theoremname}
  \theoremstyle{remark}
  \newtheorem{rem}[thm]{\protect\remarkname}
  \theoremstyle{plain}
  \newtheorem{lem}[thm]{\protect\lemmaname}

\usepackage{enumitem}
\setlist{leftmargin=*}

\makeatother

  \providecommand{\lemmaname}{Lemma}
  \providecommand{\remarkname}{Remark}
\providecommand{\theoremname}{Theorem}

\begin{document}
\global\long\def\ii{\mathrm{i}}

\global\long\def\C{\mathbb{C}}

\global\long\def\R{\mathbb{R}}

\global\long\def\N{\mathbb{N}}

\global\long\def\phi{\varphi}

\global\long\def\epsilon{\varepsilon}

\global\long\def\div{\mathrm{div}}

\global\long\def\curl{\mathrm{curl}}

\global\long\def\supp{\mathrm{supp}}

\global\long\def\Hcurl{H(\curl,\Omega)}
\global\long\def\H0curl{H_0(\curl,\Omega)}

\global\long\def\Hdiv{H(\div,\Omega)}

\global\long\def\Calpha{C^{0,\alpha}(\overline{\Omega};\C^{3})}

\global\long\def\bo{\partial\Omega}

\subjclass[2010]{35Q61, 35J57, 35B65, 35Q60}

\title[Optimal H\"older regularity for Maxwell's equations]{H\"older regularity for Maxwell's equations under minimal assumptions
on the coefficients}

\author{Giovanni S. Alberti}

\address{Department of Mathematics, University of Genoa, Via Dodecaneso 35, 16146, Italy.}
\email{alberti@dima.unige.it}

\begin{abstract}
We prove global H\"older regularity for the solutions to the time-harmonic
anisotropic Maxwell's equations, under the assumptions of H\"older continuous
coefficients. The regularity hypotheses on the coefficients are minimal.
The same estimates hold also in the case of bianisotropic material
parameters.
\end{abstract}

\keywords{Maxwell's equations, H\"older regularity, optimal regularity, Schauder
estimates, anisotropic media, bianisotropic media, Helmholtz decomposition.}

\date{November 22, 2017}

\maketitle

\section{\label{sec:Introduction}Introduction}

This paper focuses on the H\"older regularity of the solutions $E,H\in\Hcurl:=\{F\in L^{2}(\Omega;\C^{3}):\curl F\in L^{2}(\Omega;\C^{3})\}$
to the time-harmonic Maxwell's equations \cite{MONK-2003}
\begin{equation}
\left\{ \begin{array}{ll}
\curl H=\ii\omega\varepsilon E+J_{e}\quad &\mbox{in }\Omega,\\
\curl E=-\ii\omega\mu H+J_{m}\quad &\mbox{in }\Omega,\\
E\times\nu=G\times\nu\quad &\mbox{on }\partial\Omega,
\end{array}\right.\label{eq:maxwell}
\end{equation}
where $\Omega\subseteq\mathbb{R}^{3}$ is a bounded domain of class $C^{1,1}$  and the coefficients $\epsilon$ and $\mu$
belong to $L^{\infty}\left(\Omega;\mathbb{C}^{3\times3}\right)$ and
are such that for every $\eta\in\mathbb{C}^{3}$ 
\begin{equation}
\Lambda^{-1}\left|\eta\right|^{2}\le\overline{\eta}\cdot\left(\varepsilon+\overline{\varepsilon}^{T}\right)\eta,\;\Lambda^{-1}\left|\eta\right|^{2}\le\overline{\eta}\cdot\left(\mu+\overline{\mu}^{T}\right)\eta\mbox{ and }\left|\mu\right|+\left|\varepsilon\right|\le\Lambda\mbox{ a.e. in }\ensuremath{\Omega}\label{eq:Hyp-Ellip}
\end{equation}
for some $\Lambda>0$. The $3\times3$ matrix $\varepsilon$ represents
the electric permittivity and $\mu$ the magnetic permeability. The
current sources $J_{e}$ and $J_{m}$ are in $L^{2}\left(\Omega;\mathbb{C}^{3}\right)$,
the boundary value $G$ belongs to $\Hcurl$ and the frequency $\omega$
is in $\C\setminus\{0\}$. We are interested in finding (minimal)
conditions on the parameters and on the sources such that the electric
field $E$ and/or the magnetic field $H$ are H\"older continuous. The
study of the minimal regularity of $\bo$ needed goes beyond the scopes
of this work; domains with rougher boundaries are considered in \cite{AMROUCHE-BERNARDI-DAUGE-GIRAULT-1998,COSTABEL-DAUGE-2000,2001-buffa-ciarlet,2001-buffa-ciarlet-II,BUFFA-COSTABEL-SHEEN-2002,2003-buffa-costabel-dauge}.

Let us mention the main known results concerning this problem. The
H\"older continuity of the solutions under the assumption of Lipschitz
coefficients was proven in \cite{YIN-2004}. The needed regularity
of the coefficients was reduced from $W^{1,\infty}$ to $W^{1,3+\delta}$
for some $\delta>0$ in \cite{AC2014}. The case of bianisotropic
materials was treated in \cite{FERNANDES-2012,AC2014}, with similar
hypotheses and results. For related recent papers, see \cite{2014-yin-wei,2016-tsering-xiao-wei,2015-prokhorov,2016-sini}.
The arguments of all these works are based on the $H^{1}$ regularity
of the electromagnetic fields, which was first obtained in \cite{WEBER-1981}
for Lipschitz coefficients, and then in \cite{AC2014} for $W^{1,3+\delta}$
coefficients. Thus, the coefficients were always required to belong
to some Sobolev space. 

The purpose of this work is to show that it is sufficient to assume
that the coefficients are H\"older continuous. Due to the terms $\epsilon E$
and $\mu H$ in \eqref{eq:maxwell}, this is the most natural hypothesis
on $\epsilon$ and/or $\mu$, and turns out to be minimal (see Remark~\ref{rem:minimal}
below). Our approach is very different from that of \cite{AC2014},
and is based on the Helmholtz decomposition of the electromagnetic
fields, as in \cite{WEBER-1981,YIN-2004} and several related works.
However, the argument used is new, and allows to avoid any additional
differentiability of $E$ and $H$. As far as the differentiability
of the fields is concerned, it is worth mentioning that ideas similar
to those used in this work may be applied to prove the $H^{1}$ regularity
of the fields with $W^{1,3}$ coefficients \cite{alberti-capdeboscq-2016}.

Before stating the main results of this work, we need to define the weak solutions of the Maxwell system. We say that  $(E,H)\in\Hcurl^{2}$ is a
weak solution of \eqref{eq:maxwell} if
\begin{subequations}
\begin{align}
&\int_\Omega H\cdot \curl \Phi_1 \,dx = \int_\Omega\left( \ii\omega\epsilon E+J_e \right)\cdot\Phi_1\,dx,\label{eq:weak1}
\\
&\int_\Omega E\cdot\curl\Phi_2 \,dx= \int_\Omega G\cdot\curl\Phi_2-\left(\curl G+\ii\omega\mu H-J_m\right)\cdot\Phi_2\,dx,\label{eq:weak2}
\end{align}
\end{subequations}
for every $(\Phi_1,\Phi_2)\in\H0curl\times\Hcurl$,
where $\H0curl=\{F\in\Hcurl:F\times\nu=0\text{ on $\bo$}\}$. 
These identities are formally equivalent to \eqref{eq:maxwell} thanks to an integration by parts \cite[Theorem~3.29]{MONK-2003}.  
Note that for $F\in\Hcurl$, the tangential trace $F\times \nu$ on $\bo$ belongs to $H^{-\frac12}(\bo;\C^3)$ and has to be interpreted in the weak sense.

The main result of this paper regarding the joint regularity of $E$
and $H$, under the assumptions that both $\epsilon$ and $\mu$ are
H\"older continuous, reads as follows.
\begin{thm}
\label{thm:joint-holder}Assume that \eqref{eq:Hyp-Ellip} holds true
and that 
\begin{align}
 & \epsilon\in C^{0,\alpha}(\overline{\Omega};\C^{3\times3}),\qquad\left\Vert \epsilon\right\Vert _{C^{0,\alpha}(\overline{\Omega};\C^{3\times3})}\le\Lambda,\label{eq:epsi-holder}\\
 & \mu\in C^{0,\alpha}(\overline{\Omega};\C^{3\times3}),\qquad\left\Vert \mu\right\Vert _{C^{0,\alpha}(\overline{\Omega};\C^{3\times3})}\le\Lambda,\label{eq:mu-holder}
\end{align}
for some $\alpha\in(0,\frac{1}{2}]$. Take $J_{e},J_{m}\in C^{0,\alpha}(\overline{\Omega};\C^{3})$
and $G\in C^{1,\alpha}(\curl,\Omega)$, where 
\begin{align*}
 & C^{N+1,\alpha}(\curl,\Omega)=\{F\in C^{N,\alpha}(\overline{\Omega};\C^{3}):\curl F\in C^{N,\alpha}(\overline{\Omega};\C^{3})\},\qquad N\in\N,
\end{align*}
equipped with the canonical norms. Let $(E,H)\in\Hcurl^{2}$ be a
weak solution of \eqref{eq:maxwell}. Then $E,H\in\Calpha$ and 
\begin{multline*}
\left\Vert (E,H)\right\Vert _{\Calpha^{2}}\\
\le C\bigl(\left\Vert (E,H)\right\Vert _{L^{2}(\Omega;\C^{3})^{2}}+\left\Vert G\right\Vert _{C^{1,\alpha}(\curl,\Omega)}+\left\Vert (J_{e},J_{m})\right\Vert _{C^{0,\alpha}(\overline{\Omega};\C^{3})^{2}}\bigr)
\end{multline*}
 for some constant $C$ depending only on $\Omega$, $\Lambda$ and
$\omega$.
\end{thm}
The higher regularity version is given below in Theorem~\ref{thm:joint-higher}.
This result can be easily extended to treat the case of bianisotropic
materials, see Theorem~\ref{thm:bianisotropic} below.

If only one of the parameters is $C^{0,\alpha}$, for instance $\epsilon$,
the corresponding field $E$ will be H\"older continuous, provided that
$\mu$ is real. (The Campanato spaces $L^{2,\lambda}$ are defined
in Section~\ref{sec:H=0000F6lder-regularity-of}.)
\begin{thm}
\label{thm:E-holder}Assume that \eqref{eq:Hyp-Ellip} and \eqref{eq:epsi-holder}
hold true and that $\Im\mu\equiv0$. Take $J_{e},G\in C^{0,\alpha}(\overline{\Omega};\C^{3})$
with $\curl G\in L^{2,\lambda}(\Omega;\C^{3})$ for some $\lambda>1$
and $J_{m}\in L^{2,\lambda}(\Omega;\C^{3})$. Let $(E,H)\in\Hcurl^{2}$
be a weak solution of \eqref{eq:maxwell}. Then $E\in C^{0,\beta}(\overline{\Omega};\C^{3})$,
where $\beta=\min(\frac{\tilde{\lambda}-1}{2},\frac{\lambda-1}{2},\alpha)$
for some $\tilde{\lambda}\in(1,2)$ depending only on $\Omega$ and
$\Lambda$, and 
\begin{multline*}
\left\Vert E\right\Vert _{C^{0,\beta}(\overline{\Omega};\C^{3})}\\
\le C\bigl(\left\Vert E\right\Vert _{L^{2}(\Omega;\C^{3})}+\left\Vert (G,J_{e})\right\Vert _{C^{0,\alpha}(\overline{\Omega};\C^{3})^{2}}+\left\Vert (\curl G,J_{m})\right\Vert _{L^{2,\lambda}(\Omega;\C^{3})^{2}}\bigr)
\end{multline*}
 for some constant $C$ depending only on $\Omega$, $\Lambda$ and
$\omega$.
\end{thm}
The corresponding result for the H\"older regularity of $H$, assuming $\epsilon$ real and \eqref{eq:mu-holder}, is completely analogous; the details are omitted.

It is worth mentioning that the regularity assumptions
on the coefficients  given in the above theorems are indeed minimal.
\begin{rem}
\label{rem:minimal}$Let$ $\Omega=B(0,1)$ be the unit ball and take
$\alpha\in(0,1)$. Let $f\in L^{\infty}((-1,1);\R)\setminus C^{\alpha}((-1,1);\R)$
such that $\Lambda^{-1}\le f\le\Lambda$ in $(-1,1)$. Let $\epsilon$
be defined by $\epsilon(x)=f(x_{1})$. Choosing $J_{e}=(-\ii\omega,0,0)\in C^{0,\alpha}(\overline{\Omega};\C^{3})$,
observe that $E(x)=(f(x_{1})^{-1},0,0)$ and $H\equiv0$ are weak
solutions in $\Hcurl^{2}$ to
\[
\curl H=\ii\omega\varepsilon E+J_{e}\quad\mbox{in }\Omega,\qquad\curl E=-\ii\omega H\quad\mbox{in }\Omega,
\]
such that $E\notin C^{0,\alpha}(\Omega;\C^{3})$. This shows that
interior H\"older regularity cannot hold if $\epsilon$ is not H\"older
continuous, even in the simplified case where $\epsilon$ depends
only on one variable.
\end{rem}

The proofs of these results are based on the use of the scalar and vector potentials of the fields $E$ and $H$ obtained with the Helmholtz decomposition. We show that the study of the original system substantially reduces to the study of two elliptic problems for the scalar potentials, which may be treated with classical elliptic methods. This simple, and yet very powerful, idea  allows to treat the Maxwell's system as if it were elliptic. The technique developed in this paper
has then proven useful for a variety of problems, e.g.\ to state a version for Maxwell's equations of Meyers's higher integrability theorem  \cite{alberti-capdeboscq-2016}, to give asymptotic expansions of the solutions in presence of defects in the material parameters \cite{alberti-capdeboscq-2016} and for the spectral analysis of the Maxwell operator in unbounded domains \cite{alberti-etal-2017}, and has great potential for the study of other aspects of Maxwell's equations.

This paper is structured as follows. In Section~\ref{sec:Joint-H=0000F6lder-regularity}
we prove Theorem~\ref{thm:joint-holder} and discuss the corresponding
higher regularity result. Section~\ref{sec:Bi-anisotropic} is devoted
to the study of bianisotropic materials. Finally, in Section~\ref{sec:H=0000F6lder-regularity-of}
we prove Theorem~\ref{thm:E-holder}, by using standard elliptic
estimates in Campanato spaces.

\section{\label{sec:Joint-H=0000F6lder-regularity}Joint H\"older regularity
of $E$ and $H$}

Since regularity properties are local, without loss of generality in the rest of the paper we  assume that $\Omega$ is connected and
simply connected and that its boundary
$\bo$ is connected.

\subsection{Preliminary results}

We start by recalling the Helmholtz decomposition of a vector field. 
\begin{lem}[{\cite[Theorem 6.1]{AMROUCHE-SELOULA-2013}, \cite[Section 3.5]{AMROUCHE-BERNARDI-DAUGE-GIRAULT-1998}}]
\label{lem:helmholtz}Take $F\in L^{2}(\Omega;\C^{3})$.
\begin{enumerate}
\item There exist $q\in H_{0}^{1}(\Omega;\C)$ and $\Phi\in H^{1}(\Omega;\C^{3})$
such that
\[
F=\nabla q+\curl\Phi\quad\text{in }\Omega,
\]
$\div\Phi=0$ in $\Omega$ and $\Phi\cdot\nu=0$ on $\bo$.
\item There exist $q\in H^{1}(\Omega;\C)$ and $\Phi\in H^{1}(\Omega;\C^{3})$
such that
\[
F=\nabla q+\curl\Phi\quad\text{in }\Omega,
\]
$\div\Phi=0$ in $\Omega$ and $\Phi\times\nu=0$ on $\bo$.
\end{enumerate}
In both cases, there exists $C>0$ depending only on $\Omega$ such
that
\[
\left\Vert \Phi\right\Vert _{H^{1}(\Omega;\C^{3})}\le C\left\Vert F\right\Vert _{L^{2}(\Omega;\C^{3})}.
\]

\end{lem}
We shall need the following key estimate.
\begin{lem}[\cite{AMROUCHE-SELOULA-2013}]
\label{lem:friedrichs}Take $p\in(1,\infty)$ and $F\in L^{p}(\Omega;\C^{3})$
such that $\curl F\in L^{p}(\Omega;\C^{3})$, $\div F\in L^{p}(\Omega;\C)$
and either $F\cdot\nu=0$ or $F\times\nu=0$ on $\bo$. Then $F\in W^{1,p}(\Omega;\C^{3})$
and
\[
\left\Vert F\right\Vert _{W^{1,p}(\Omega;\C^{3})}\le C\bigl(\left\Vert \curl F\right\Vert _{L^{p}(\Omega;\C^{3})}+\left\Vert \div F\right\Vert _{L^{p}(\Omega;\C)}\bigr),
\]
for some $C>0$ depending only on $\Omega$ and $p$.
\end{lem}

\subsection{Proof of Theorem~\ref{thm:joint-holder}}

With an abuse of notation, several positive constants depending only
on $\Omega$, $\Lambda$ and $\omega$ will be denoted by the same
letter $C$.

First, we express $E-G$ and $H$ by means of scalar and vector potentials
by using Lemma~\ref{lem:helmholtz}: there exist $q_{E}\in H_{0}^{1}(\Omega;\C)$,
$q_{H}\in H^{1}(\Omega;\C)$ and $\Phi_{E},\Phi_{H}\in H^{1}(\Omega;\C^{3})$
such that
\begin{equation}
E-G=\nabla q_{E}+\curl\Phi_{E}\quad\text{in }\Omega,\qquad H=\nabla q_{H}+\curl\Phi_{H}\quad\text{in }\Omega,\label{eq:E,H helmholtz}
\end{equation}
and
\begin{equation}
\left\{ \begin{array}{ll}
\div\Phi_{E}=0\quad&\mbox{in }\Omega,\\
\Phi_{E}\cdot\nu=0\quad&\mbox{on }\partial\Omega,
\end{array}\right.\qquad\left\{ \begin{array}{ll}
\div\Phi_{H}=0\quad&\mbox{in }\Omega,\\
\Phi_{H}\times\nu=0\quad&\mbox{on }\partial\Omega.
\end{array}\right.\label{eq:PhiE-PhiH}
\end{equation}
Moreover, there exists $C>0$ depending only on $\Omega$ such that
\begin{equation}
\left\Vert \Phi_{E}\right\Vert _{H^{1}(\Omega;\C^{3})}\le C\left\Vert (E,G)\right\Vert _{L^{2}(\Omega;\C^{3})^{2}},\qquad\left\Vert \Phi_{H}\right\Vert _{H^{1}(\Omega;\C^{3})}\le C\left\Vert H\right\Vert _{L^{2}(\Omega;\C^{3})}.\label{eq:PhiE-PhiH-W1p}
\end{equation}

By Lemma~\ref{lem:friedrichs}, the vector potentials enjoy additional
regularity.
\begin{lem}
\label{lem:Phi-W2p}Assume that \eqref{eq:Hyp-Ellip} holds true and
take $p\in[2,\infty)$. Take $J_{e},J_{m}\in L^{p}(\Omega;\C^{3})$
and $G\in W^{1,p}(\curl,\Omega)$. Let $(E,H)\in W^{1,p}(\curl,\Omega)^{2}$
be a weak solution of \eqref{eq:maxwell}, where
\[
W^{1,p}(\curl,\Omega):=\{F\in L^{p}(\Omega;\C^{3}):\curl F\in L^{p}(\Omega;\C^{3})\},
\]
equipped with the canonical norm. Then $\curl\Phi_{E},\curl\Phi_{H}\in W^{1,p}(\Omega;\C^{3})$
and
\begin{align*}
 & \left\Vert \curl\Phi_{E}\right\Vert _{W^{1,p}(\Omega;\C^{3})}\le C\left\Vert (H,J_{m},\curl G)\right\Vert _{L^{p}(\Omega;\C^{3})^{3}},\\
 & \left\Vert \curl\Phi_{H}\right\Vert _{W^{1,p}(\Omega;\C^{3})}\le C\left\Vert (E,J_{e})\right\Vert _{L^{p}(\Omega;\C^{3})^{2}},
\end{align*}
for some constant $C$ depending only on $\Omega$, $\Lambda$ and
$\omega$.\end{lem}
\begin{proof}
Set $\Psi_{E}:=\curl\Phi_{E}$. 
Observe that for every test function $\Phi\in C^\infty(\overline\Omega;\C^3)$  we have
\[
\begin{split}
\int_\Omega\! \curl(\nabla q_E)\cdot\Phi-\nabla q_E\cdot\curl\Phi \,dx &= \int_\Omega\!  q_E\div(\curl\Phi) \,dx - \int_{\bo}\!q_E \curl\Phi\cdot\nu\,ds\\
& =0,
\end{split}
\]
which implies that $\nabla q_E\times\nu =0$ on $\bo$  (\cite[Theorem~3.33]{MONK-2003}). Hence, by \eqref{eq:E,H helmholtz} and the
third equation of \eqref{eq:maxwell} we obtain
\[
\Psi_{E}\times\nu=(\curl\Phi_{E})\times\nu=(E-G)\times\nu-\nabla q_{E}\times\nu=0\quad\text{on }\bo,
\]
Thus, using the first equation
of \eqref{eq:maxwell} and the identities $\curl\nabla=0$ and $\div\,\curl=0$
we obtain 
\begin{equation}
\left\{ \begin{array}{ll}
\curl\Psi_{E}=-\ii\omega\mu H+J_{m}-\curl G\quad&\text{in }\Omega,\\
\div\Psi_{E}=0\quad&\mbox{in }\Omega,\\
\Psi_{E}\times\nu=0\quad&\mbox{on }\partial\Omega.
\end{array}\right.\label{eq:PsiE}
\end{equation}
Therefore, by Lemma~\ref{lem:friedrichs} we have that $\curl\Phi_{E}\in W^{1,p}(\Omega;\C^{3})$
and
\[
\left\Vert \curl\Phi_{E}\right\Vert _{W^{1,p}(\Omega;\C^{3})}\le C\left\Vert (H,J_{m},\curl G)\right\Vert _{L^{p}(\Omega;\C^{3})^{3}}.
\]

The proof for $\Phi_{H}$ is similar, only the boundary conditions
have to be handled in a different way. As above, set $\Psi_{H}:=\curl\Phi_{H}$. 
For every test function $\varphi\in C^\infty(\overline\Omega;\C)$  we have
\[
\int_\Omega \Psi_H\cdot\nabla\varphi -\varphi\, \div\Psi_H \,dx = \int_\Omega \curl\Phi_H\cdot\nabla\varphi\,dx =0
\]
since $\curl\nabla=0$ and $\Phi_H\times\nu=0$ on $\bo$. This identity implies that
\[
\Psi_{H}\cdot\nu=0\quad\text{on }\bo.
\]
Moreover $\div\Psi_{H}=0$ in $\Omega$ and using the second equation
of \eqref{eq:maxwell} we obtain $\curl\Psi_{H}=\ii\omega\varepsilon E+J_{e}\in L^{p}(\Omega;\C^{3})$.
Therefore, by Lemma~\ref{lem:friedrichs} we have that $\curl\Phi_{H}\in W^{1,p}(\Omega;\C^{3})$
and
\[
\left\Vert \curl\Phi_{H}\right\Vert _{W^{1,p}(\Omega;\C^{3})}\le C\left\Vert (E,J_{e})\right\Vert _{L^{p}(\Omega;\C^{3})^{2}}.
\]
This concludes the proof.
\qed
\end{proof}
We are now in a position to prove Theorem~\ref{thm:joint-holder}.

\vspace{5pt}\noindent\emph{Proof of Theorem~\ref{thm:joint-holder}$\,$}
The proof is divided into two steps. 

\emph{Step 1. $W^{1,6}$-regularity of the scalar potentials.} By
Lemma~\ref{lem:Phi-W2p} with $p=2$ and the Sobolev embedding theorem,
we have that $\curl\Phi_{E},\curl\Phi_{H}\in L^{6}(\Omega;\C^{3})$
and
\begin{equation}
\left\Vert (\curl\Phi_{E},\curl\Phi_{H})\right\Vert _{L^{6}(\Omega;\C^{3})^{2}}\le C\left\Vert (E,H,\curl G,J_{e},J_{m})\right\Vert _{L^{2}(\Omega;\C^{3})^{5}}.\label{eq:PhiE-PhiH-W16}
\end{equation}

By \eqref{eq:E,H helmholtz} and \eqref{eq:weak1} with $\Phi_1=\nabla\varphi$ for $\varphi\in H^1_0(\Omega;\C)$ (arguing as in the first part of the proof of Lemma~\ref{lem:Phi-W2p}, we have $\Phi_1\in\H0curl$) we obtain
\[
  \int_\Omega\left( \ii\omega\epsilon (G+\curl\Phi_E+\nabla q_E)+J_e \right)\cdot\nabla\varphi\,dx = 0,\qquad \varphi\in H^1_0(\Omega;\C).
\]
In other words,  $q_{E}$ is a weak
solution of 
\begin{equation}
\left\{ \begin{array}{ll}
-\div(\epsilon\nabla q_{E})=\div(\epsilon G+\epsilon\curl\Phi_{E}-\ii\omega^{-1}J_{e})\quad&\mbox{in }\Omega,\\
q_{E}=0\quad&\mbox{on }\bo.
\end{array}\right.\label{eq:PDE-qE}
\end{equation}
Similarly, using \eqref{eq:E,H helmholtz} and \eqref{eq:weak2} with $\Phi_2=\nabla\varphi$ for $\varphi\in H^1(\Omega;\C)$
we have
\[
 \int_\Omega \left(\curl G+\ii\omega\mu (\nabla q_H+\curl\Phi_H)-J_m\right)\cdot\nabla\varphi\,dx = 0,\qquad \varphi\in H^1(\Omega;\C).
\]
In other words, $q_H$ is a weak solution of
\begin{equation}
\left\{ \begin{array}{ll}
-\div(\mu\nabla q_{H})=\div(\mu\curl\Phi_{H}+\ii\omega^{-1}J_{m}-\ii\omega^{-1}\curl G)\quad&\mbox{in }\Omega,\\
-(\mu\nabla q_{H})\cdot\nu=(\mu\curl\Phi_{H}+\ii\omega^{-1}J_{m}-\ii\omega^{-1}\curl G)\cdot\nu\quad&\mbox{on }\bo.
\end{array}\right.\label{eq:PDE-qH}
\end{equation}
Therefore, by the $L^{p}$ theory for elliptic equations with complex
coefficients (see, e.g., \cite[Theorem 1]{AUSCHER-QAFSAOUI-2002})
applied to the above boundary value problems, we obtain $\nabla q_{E},\nabla q_{H}\in L^{6}(\Omega;\C^{3})$
and
\begin{multline}
\left\Vert \nabla q_{E},\nabla q_{H})\right\Vert _{L^{6}(\Omega;\C^{3})^{2}}
\le C\bigl(\left\Vert (\curl\Phi_{E},\curl\Phi_{H})\right\Vert _{L^{6}(\Omega;\C^{3})^{2}}\bigr.
\\\bigl.+\left\Vert G\right\Vert _{W^{1,6}(\curl,\Omega)}+\left\Vert (J_{e},J_{m})\right\Vert _{L^{6}(\Omega;\C^{3})^{2}}\bigr).\label{eq:qE-qH-W16}
\end{multline}

\emph{Step 2. $C^{1,\alpha}$-regularity of the scalar potentials.}
Combining \eqref{eq:PhiE-PhiH-W16} and \eqref{eq:qE-qH-W16} we have
$E,H\in L^{6}(\Omega;\C^{3})$ and
\begin{multline*}
\left\Vert (E,H)\right\Vert _{L^{6}(\Omega;\C^{3})^{2}}\\
\le C\bigl(\left\Vert (E,H)\right\Vert _{L^{2}(\Omega;\C^{3})^{2}}+\left\Vert G\right\Vert _{W^{1,6}(\curl,\Omega)}+\left\Vert (J_{e},J_{m})\right\Vert _{L^{6}(\Omega;\C^{3})^{2}}\bigr).
\end{multline*}
Thus, by Lemma~\ref{lem:Phi-W2p} with $p=6$ we obtain $\curl\Phi_{E},\curl\Phi_{H}\in W^{1,6}(\Omega;\C^{3})$
and
\begin{multline*}
\left\Vert (\curl\Phi_{E},\curl\Phi_{H})\right\Vert _{W^{1,6}(\Omega;\C^{3})^{2}}\\
\le C\bigl(\left\Vert (E,H)\right\Vert _{L^{2}(\Omega;\C^{3})^{2}}+\left\Vert G\right\Vert _{W^{1,6}(\curl,\Omega)}+\left\Vert (J_{e},J_{m})\right\Vert _{L^{6}(\Omega;\C^{3})^{2}}\bigr).
\end{multline*}
By the Sobolev embedding theorem, this implies $\curl\Phi_{E},\curl\Phi_{H}\in C^{0,\frac{1}{2}}(\overline{\Omega};\C^{3})$
and
\begin{multline*}
\left\Vert (\curl\Phi_{E},\curl\Phi_{H})\right\Vert _{C^{0,\frac{1}{2}}(\overline{\Omega};\C^{3})^{2}}\\
\le C\bigl(\left\Vert (E,H)\right\Vert _{L^{2}(\Omega;\C^{3})^{2}}+\left\Vert G\right\Vert _{W^{1,6}(\curl,\Omega)}+\left\Vert (J_{e},J_{m})\right\Vert _{L^{6}(\Omega;\C^{3})^{2}}\bigr).
\end{multline*}
In view of \eqref{eq:epsi-holder}-\eqref{eq:mu-holder}, by applying
classical Schauder estimates for elliptic systems \cite{GIAQUINTA-MARTINAZZI-2005,2008-morrey}
to \eqref{eq:PDE-qE} and \eqref{eq:PDE-qH} we obtain
\begin{multline*}
\left\Vert (q_{E},q_{H})\right\Vert _{C^{1,\alpha}(\overline{\Omega};\C)^{2}}\\
\le C\bigl(\left\Vert (E,H)\right\Vert _{L^{2}(\Omega;\C^{3})^{2}}+\left\Vert G\right\Vert _{C^{1,\alpha}(\curl,\Omega)}+\left\Vert (J_{e},J_{m})\right\Vert _{\Calpha^{2}}\bigr).
\end{multline*}
Finally, the result follows from \eqref{eq:E,H helmholtz} and the
last two estimates.\qed

\subsection{Higher regularity}

The proof of Theorem~\ref{thm:joint-holder} is based on the regularity
of the scalar and vector potentials of the electric and magnetic fields.
In particular, the regularity of $\Phi_{E}$ and $\Phi_{H}$ follows
from Lemma \ref{lem:friedrichs}, while the regularity of $q_{E}$
and $q_{H}$ follows from standard $L^{p}$ and Schauder estimates
for elliptic systems. Since all these estimates admit higher regularity
generalisations \cite{AMROUCHE-SELOULA-2013,2008-morrey}, by following
the argument outlined above we immediately obtain the corresponding
higher regularity result.
\begin{thm}
\label{thm:joint-higher}Assume that \eqref{eq:Hyp-Ellip} holds true,
that $\bo$ is of class $C^{N+1,1}$ and that
\[
\epsilon,\mu\in C^{N,\alpha}(\overline{\Omega};\C^{3\times3}),\qquad\left\Vert (\epsilon,\mu)\right\Vert _{C^{N,\alpha}(\overline{\Omega};\C^{3\times3})^{2}}\le\Lambda,
\]
for $\alpha\in(0,\frac{1}{2}]$ and $N\in\N$. Take $J_{e},J_{m}\in C^{N,\alpha}(\overline{\Omega};\C^{3})$
and $G\in C^{N+1,\alpha}(\curl,\Omega)$. Let $(E,H)\in\Hcurl^{2}$
be a weak solution of \eqref{eq:maxwell}. Then $E,H\in C^{N,\alpha}(\overline{\Omega};\C^{3})$
and 
\begin{multline*}
\left\Vert (E,H)\right\Vert _{C^{N,\alpha}(\overline{\Omega};\C^{3})}\\
\le C\bigl(\left\Vert (E,H)\right\Vert _{L^{2}(\Omega;\C^{3})^{2}}+\left\Vert G\right\Vert _{C^{N+1,\alpha}(\curl,\Omega)}+\left\Vert (J_{e},J_{m})\right\Vert _{C^{N,\alpha}(\overline{\Omega};\C^{3})^{2}}\bigr)
\end{multline*}
for some constant $C$ depending only on $\Omega$, $\Lambda$, $\omega$
and $N$.
\end{thm}

\section{\label{sec:Bi-anisotropic} The case of bianisotropic materials}

In this section, we investigate the H\"older regularity of the solutions
of the following problem 
\begin{equation}
\left\{ \begin{array}{ll}
\curl H=\ii\omega\left(\varepsilon E+\xi H\right)+J_{e}\quad&\mbox{in }\Omega,\\
\curl E=-\ii\omega\left(\zeta E+\mu H\right)+J_{m}\quad&\mbox{in }\Omega,\\
E\times\nu=G\times\nu\quad&\text{on }\bo.
\end{array}\right.\label{eq:maxwell for bi-anisotropic}
\end{equation}
In this general case, \eqref{eq:Hyp-Ellip} is not sufficient to ensure
ellipticity. As we will see, the leading order coefficient of the
coupled elliptic system corresponding to \eqref{eq:PDE-qE}-\eqref{eq:PDE-qH}
is 
\[
A=A_{ij}^{\alpha\beta}=\left[\begin{array}{cccc}
\Re\varepsilon & -\Im\varepsilon & \Re\xi & -\Im\xi\\
\Im\varepsilon & \Re\varepsilon & \Im\xi & \Re\xi\\
\Re\zeta & -\Im\zeta & \Re\mu & -\Im\mu\\
\Im\zeta & \Re\zeta & \Im\mu & \Re\mu
\end{array}\right],
\]
where the Latin indices $i,j=1,\dots,4$ identify the different $3\times3$
block sub-matrices, whereas the Greek letters $\alpha,\beta=1,2,3$
span each of these $3\times3$ block sub-matrices. We assume that
$A$ is in $L^{\infty}(\Omega;\mathbb{R})^{12\times12}$ and that
satisfies a strong Legendre condition (as in \cite{CHEN-WU-1998,GIAQUINTA-MARTINAZZI-2005}),
namely 
\begin{equation}
A_{ij}^{\alpha\beta}\eta_{\alpha}^{i}\eta_{\beta}^{j}\ge\Lambda^{-1}\left|\eta\right|^{2},\;\eta\in\mathbb{R}^{12}\quad\mbox{ and}\quad\bigl|A_{ij}^{\alpha\beta}\bigr|\le\Lambda\qquad\mbox{ a.e. in \ensuremath{\Omega}}\label{eq:strong legendre condition}
\end{equation}
for some $\Lambda>0$. This condition is satisfied by a large class
of materials, including chiral materials and all natural materials
\cite[Lemma 10 and Remark 11]{AC2014}. Moreover, generalising the
regularity assumptions given in \eqref{eq:epsi-holder}-\eqref{eq:mu-holder},
we suppose that
\begin{equation}
\epsilon,\xi,\zeta,\mu\in C^{0,\alpha}(\overline{\Omega};\C^{3\times3}),\qquad\left\Vert (\epsilon,\xi,\zeta,\mu)\right\Vert _{C^{0,\alpha}(\overline{\Omega};\C^{3\times3})^{4}}\le\Lambda\label{eq:bianisotropic-holder}
\end{equation}
for some $\alpha\in(0,\frac{1}{2}]$.

The main result of this section reads as follows.
\begin{thm}
\label{thm:bianisotropic}Assume that \eqref{eq:strong legendre condition}
and \eqref{eq:bianisotropic-holder} hold true. Take $J_{e},J_{m}\in C^{0,\alpha}(\overline{\Omega};\C^{3})$
and $G\in C^{1,\alpha}(\curl,\Omega)$. Let $(E,H)\in\Hcurl^{2}$
be a weak solution of \eqref{eq:maxwell for bi-anisotropic}. Then
$E,H\in\Calpha$ and 
\begin{multline*}
\left\Vert (E,H)\right\Vert _{\Calpha^{2}}\\
\le C\bigl(\left\Vert (E,H)\right\Vert _{L^{2}(\Omega;\C^{3})^{2}}+\left\Vert G\right\Vert _{C^{1,\alpha}(\curl,\Omega)}+\left\Vert (J_{e},J_{m})\right\Vert _{C^{0,\alpha}(\overline{\Omega};\C^{3})^{2}}\bigr)
\end{multline*}
 for some constant $C$ depending only on $\Omega$, $\Lambda$ and
$\omega$.\end{thm}
\begin{proof}
The main ingredients are the same used for the proof of Theorem~\ref{thm:joint-holder}.
In particular, the regularity result on the vector potentials $\Phi_{E}$
and $\Phi_{H}$ of $E-G$ and $H$ given in Lemma~\ref{lem:Phi-W2p}
holds true also in this case. The only difference lies in the fact
that, since the bianisotropy mixes the electric and magnetic properties,
the corresponding estimates will be
\begin{equation}
\left\Vert (\curl\Phi_{E},\curl\Phi_{H})\right\Vert _{W^{1,p}(\Omega;\C^{3})^{2}}\le C\left\Vert (E,H,J_{e},J_{m},\curl G)\right\Vert _{L^{p}(\Omega;\C^{3})^{5}}.\label{eq:bian1}
\end{equation}
Similarly, as far as the scalar potentials are concerned, the two
equations \eqref{eq:PDE-qE}-\eqref{eq:PDE-qH} become a fully coupled
elliptic system, namely
\[
\begin{array}{l}
-\div(\epsilon\nabla q_{E}+\xi\nabla q_{H})=\div(\epsilon G+\epsilon\curl\Phi_{E}+\xi\curl\Phi_{H}-\ii\omega^{-1}J_{e}),\\
-\div(\zeta\nabla q_{E}+\mu\nabla q_{H})=\div(\zeta G+\zeta\curl\Phi_{E}+\mu\curl\Phi_{H}+\ii\omega^{-1}(J_{m}-\curl G)),
\end{array}
\]
in $\Omega$, augmented with the boundary conditions
\[
\begin{array}{l}
 q_{E}=0,\\
-(\zeta\nabla q_{E}+\mu\nabla q_{H})\cdot\nu=(\zeta G+\zeta\curl\Phi_{E}+\mu\curl\Phi_{H}+\ii\omega^{-1}(J_{m}-\curl G))\cdot\nu.
\end{array}
\]
on $\bo$. More precisely, the weak form of this system reads{\small
\begin{align*}
&\int_\Omega (\epsilon\nabla q_{E}+\xi\nabla q_{H})\cdot\nabla\varphi_1\,dx=\int_\Omega  (\epsilon G+\epsilon\curl\Phi_{E}+\xi\curl\Phi_{H}-\ii\omega^{-1}J_{e})\cdot\nabla\varphi_1\,dx, \\
&\int_\Omega (\zeta\nabla q_{E}+\mu\nabla q_{H})\cdot\nabla\varphi_2\,dx=\int_\Omega  \biggl(\zeta G+\zeta\curl\Phi_{E}+\mu\curl\Phi_{H}+\frac{\curl G-J_{m}}{\ii\omega}\biggr)\cdot\nabla\varphi_2\,dx,
\end{align*}
}for every $(\varphi_1,\varphi_2)\in H^1_0(\Omega;\C)\times H^1(\Omega;\C)$. 
By \eqref{eq:strong legendre condition}, this system is strongly
elliptic, and since the coefficients are H\"older continuous, both the
$L^{p}$ theory and the Schauder theory are applicable \cite[Theorem 6.4.8]{2008-morrey}.

We now present a quick sketch of the proof, which follows exactly
the same structure of the proof of Theorem~\ref{thm:joint-holder}.
By \eqref{eq:bian1} with $p=2$ we first deduce that $\curl\Phi_{E}$
and $\curl\Phi_{H}$ belong to $L^{6}$. Thus, by applying the $L^{p}$
theory to the elliptic system above, we deduce that the scalar potentials
are in $W^{1,6}$. By \eqref{eq:E,H helmholtz}, this implies that
$E$ and $H$ are in $L^{6}$. Using again \eqref{eq:bian1} with
$p=6$ we deduce that $\curl\Phi_{E}$ and $\curl\Phi_{H}$ are H\"older
continuous. Finally, by the Schauder estimates we deduce that $\nabla q_{E}$
and $\nabla q_{H}$ are H\"older continuous. The corresponding norm
estimate follows as in the proof of Theorem~\ref{thm:joint-holder}.
\qed\end{proof}

\section{\label{sec:H=0000F6lder-regularity-of}H\"older regularity of the electric
field $E$}

The proof of Theorem~\ref{thm:E-holder} is based on standard elliptic
estimates in Campanato spaces \cite{CAMPANATO-1980}, which we now
introduce. For $\lambda\ge0$, let $L^{2,\lambda}(\Omega;\C)$ be
the Banach space of functions $u\in L^{2}\left(\Omega;\C\right)$
such that 
\[
[u]_{2,\lambda;\Omega}^{2}:=\!\!\sup_{x\in\Omega,0<\rho<{\rm diam}\Omega}\rho^{-\lambda}\int_{\Omega(x,\rho)}\Bigl|u(y)-\frac{1}{|\Omega(x,\rho)|}\int_{\Omega(x,\rho)}u(z)\, dz\Bigr|^{2}\, dy<\infty,
\]
 where $\Omega(x,\rho)=\Omega\cap\{y\in\mathbb{R}^{3}:\left|y-x\right|<\rho\}$.
The space $L^{2,\lambda}(\Omega;\C)$ is naturally equipped with the
norm 
\[
\left\Vert u\right\Vert _{L^{2,\lambda}(\Omega;\C)}=\left\Vert u\right\Vert _{L^{2}(\Omega;\C)}+[u]_{2,\lambda;\Omega}.
\]

We shall use the following standard properties.
\begin{lem}[{\cite[Chapter 1]{TROIANIELLO-1987}}]
\label{lem:campanato_properties}Take $\lambda\ge0$ and $p\in[2,\infty)$. 
\begin{enumerate}
\item If $\lambda\in(3,5)$ then $L^{2,\lambda}\left(\Omega;\C\right)\cong C^{0,\frac{\lambda-3}{2}}\left(\overline{\Omega};\C\right)$. 
\item If $\lambda<3$, $u\in L^{2}(\Omega;\C)$ and $\nabla u\in L^{2,\lambda}\left(\Omega;\C^{3}\right)$
then $u\in L^{2,2+\lambda}\left(\Omega;\C\right)$, and the embedding
is continuous. 
\item The embedding $L^{p}\left(\Omega;\C\right)\hookrightarrow L^{2,3\frac{p-2}{p}}\left(\Omega;\C\right)$
is continuous. 
\end{enumerate}
\end{lem}
We now state the regularity result regarding Campanato estimates we
will use.
\begin{lem}[{\cite[Theorem 2.19]{TROIANIELLO-1987}}]
\label{lem:campanato_regularity} Assume that \eqref{eq:Hyp-Ellip}
holds and that $\Im\mu\equiv0$. There exists $\tilde{\lambda}\in(1,2)$
depending only on $\Omega$ and $\Lambda$ such that if $F\in L^{2,\lambda}\left(\Omega;\mathbb{C}^{3}\right)$
for some $\lambda\in[0,\tilde{\lambda}]$, and $u\in H^{1}\left(\Omega;\mathbb{C}\right)$
satisfies 
\[
\left\{ \begin{array}{ll}
\div(\mu\nabla u)=\div F\quad&\mbox{in }\Omega,\\
\mu\nabla u\cdot\nu=F\cdot\nu\quad&\mbox{on }\partial\Omega,
\end{array}\right.
\]
 then $\nabla u\in L^{2,\lambda}\left(\Omega;\mathbb{C}^{3}\right)$
and 
\begin{equation}
\left\Vert \nabla u\right\Vert _{L^{2,\lambda}\left(\Omega;\mathbb{C}^{3}\right)}\le C\left\Vert F\right\Vert _{L^{2,\lambda}\left(\Omega;\mathbb{C}^{3}\right)}\label{eq:estimate campanato}
\end{equation}
 for some constant $C$ depending only on $\Omega$ and $\Lambda$.
\end{lem}
We shall need the following generalisation of Lemma~\ref{lem:friedrichs}
to the case of Campanato estimates. For a proof, see the second part
of the proof of \cite[Theorem 3.4]{YIN-2004}.
\begin{lem}
\label{lem:friedrichs-campanato}Take $\lambda\in[0,2)$ and $F\in L^{2}(\Omega;\C^{3})$
such that $\curl F\in L^{2,\lambda}(\Omega;\C^{3})$, $\div F\in L^{2,\lambda}(\Omega;\C)$
and $F\times\nu=0$ on $\bo$. Then $\nabla F\in L^{2,\lambda}(\Omega;\C^{3})$
and 
\[
\left\Vert \nabla F\right\Vert _{L^{2,\lambda}(\Omega;\C^{3})}\le C\bigl(\left\Vert \curl F\right\Vert _{L^{2,\lambda}(\Omega;\C^{3})}+\left\Vert \div F\right\Vert _{L^{2,\lambda}(\Omega;\C)}\bigr),
\]
for some $C>0$ depending only on $\Omega$ and $\lambda$.
\end{lem}
We are now in a position to prove Theorem~\ref{thm:E-holder}.

\vspace{5pt}\noindent\emph{Proof of Theorem~\ref{thm:E-holder}$\,$}
With an abuse of notation, several positive constants depending only
on $\Omega$, $\Lambda$ and $\omega$ will be denoted by the same
letter $C$.

Write $E-G$ and $H$ in terms of scalar and vector potentials $(q_{E},\Phi_{E})$
and $(q_{H},\Phi_{H})$, as in \eqref{eq:E,H helmholtz}. By Lemma~\ref{lem:Phi-W2p}
and the Sobolev embedding theorem $\curl\Phi_{H}\in L^{6}(\Omega;\C^{3})$
and $\left\Vert \curl\Phi_{H}\right\Vert _{L^{6}(\Omega;\C^{3})}\le C\left\Vert (E,J_{e})\right\Vert _{L^{2}(\Omega;\C^{3})^{2}}$.
Thus, by Lemma~\ref{lem:campanato_properties}, part (3), we have
that $\curl\Phi_{H}\in L^{2,2}(\Omega;\C^{3})$ and
\[
\left\Vert \curl\Phi_{H}\right\Vert _{L^{2,2}(\Omega;\C^{3})}\le C\left\Vert (E,J_{e})\right\Vert _{L^{2}(\Omega;\C^{3})^{2}}.
\]
Therefore, applying Lemma~\ref{lem:campanato_regularity} to \eqref{eq:PDE-qH}
we obtain that $\nabla q_{H}\in L^{2,\min(\tilde{\lambda},\lambda)}(\Omega;\C^{3})$
and
\[
\left\Vert \nabla q_{H}\right\Vert _{L^{2,\min(\tilde{\lambda},\lambda)}(\Omega;\C^{3})}\le C(\left\Vert (E,J_{e})\right\Vert _{L^{2}(\Omega;\C^{3})^{2}}+\left\Vert (\curl G,J_{m})\right\Vert _{L^{2,\lambda}(\Omega;\C^{3})^{2}}).
\]
Combining the last two inequalities we obtain the estimate
\[
\left\Vert H\right\Vert _{L^{2,\min(\tilde{\lambda},\lambda)}(\Omega;\C^{3})}\le C(\left\Vert (E,J_{e})\right\Vert _{L^{2}(\Omega;\C^{3})^{2}}+\left\Vert (\curl G,J_{m})\right\Vert _{L^{2,\lambda}(\Omega;\C^{3})^{2}}).
\]
As a consequence, applying Lemma~\ref{lem:friedrichs-campanato}
to $\Psi_{E}=\curl\Phi_{E}$, by \eqref{eq:PsiE} and the fact that
$L^{\infty}$ is a multiplier space for $L^{2,\min(\tilde{\lambda},\lambda)}$,
we obtain that $\nabla\curl\Phi_{E}\in L^{2,\min(\tilde{\lambda},\lambda)}(\Omega;\C^{3})$
and
\[
\left\Vert \nabla\curl\Phi_{E}\right\Vert _{L^{2,\min(\tilde{\lambda},\lambda)}(\Omega;\C^{3})}\le C(\left\Vert (E,J_{e})\right\Vert _{L^{2}(\Omega;\C^{3})^{2}}+\left\Vert (\curl G,J_{m})\right\Vert _{L^{2,\lambda}(\Omega;\C^{3})^{2}}).
\]
Hence, by Lemma~\ref{lem:campanato_properties}, part (2), and \eqref{eq:PhiE-PhiH-W1p}
we have that $\curl\Phi_{E}\in L^{2,2+\min(\tilde{\lambda},\lambda)}(\Omega;\C^{3})$
and
\begin{multline*}
\left\Vert \curl\Phi_{E}\right\Vert _{L^{2,2+\min(\tilde{\lambda},\lambda)}(\Omega;\C^{3})}\\
\le C(\left\Vert (E,G,J_{e})\right\Vert _{L^{2}(\Omega;\C^{3})^{3}}+\left\Vert (\curl G,J_{m})\right\Vert _{L^{2,\lambda}(\Omega;\C^{3})^{2}}).
\end{multline*}
Then, by Lemma~\ref{lem:campanato_properties}, part (1), we obtain
that $\curl\Phi_{E}\in C^{0,\frac{\min(\tilde{\lambda},\lambda)-1}{2}}(\overline{\Omega};\C^{3})$
and
\begin{multline*}
\left\Vert \curl\Phi_{E}\right\Vert _{C^{0,\frac{\min(\tilde{\lambda},\lambda)-1}{2}}(\overline{\Omega};\C^{3})}\\
\le C(\left\Vert (E,G,J_{e})\right\Vert _{L^{2}(\Omega;\C^{3})^{3}}+\left\Vert (\curl G,J_{m})\right\Vert _{L^{2,\lambda}(\Omega;\C^{3})^{2}}).
\end{multline*}
By \eqref{eq:epsi-holder}, classical Schauder estimates applied to
\eqref{eq:PDE-qE} yield $\nabla q_{E}\in C^{0,\beta}(\overline{\Omega};\C^{3})$,
where $\beta=\min(\frac{\tilde{\lambda}-1}{2},\frac{\lambda-1}{2},\alpha)$,
and
\begin{multline*}
\left\Vert \nabla q_{E}\right\Vert _{C^{0,\beta}(\overline{\Omega};\C^{3})}\\
\le C(\left\Vert E\right\Vert _{L^{2}(\Omega;\C^{3})}+\left\Vert (\curl G,J_{m})\right\Vert _{L^{2,\lambda}(\Omega;\C^{3})^{2}}+\left\Vert (G,J_{e})\right\Vert _{\Calpha^{2}}).
\end{multline*}
Finally, combining the last two estimates yields the result.
\qed

\bibliographystyle{plain}
\phantomsection\addcontentsline{toc}{section}{\refname}\bibliography{biblio}

\end{document}